\documentclass{amsart}

\usepackage{graphicx}
\usepackage{amssymb, amsmath, amsthm}
\usepackage{color}
\usepackage{hyperref}

\newtheorem{lemma}{Lemma}
\newtheorem{proposition}[lemma]{Proposition}
\newtheorem{theorem}[lemma]{Theorem}
\newtheorem{claim}[lemma]{Claim}
\newtheorem{definition}[lemma]{Definition}
\newtheorem{corollary}[lemma]{Corollary}

\newcommand{\R}{\mathbb R}
\renewcommand{\S}{\mathbb S}

\newcommand{\pd}{\partial}

\newcommand{\ve}{\varepsilon}

\newcommand{\rncyl}{r_{n}}
\newcommand{\cyl}{\mathcal C}
\newcommand{\pla}{\mathcal P}

\begin{document}

\title{Shrinking doughnuts via variational methods}


\author{Gregory Drugan}
\address{Oregon Episcopal School, 6300 SW Nicol Road, Portland, OR 97223}
\email{drugan.math@gmail.com}

\author{Xuan Hien Nguyen}
\address{Department of Mathematics, Iowa State University, Ames, IA 50011}
\email{xhnguyen@iastate.edu}

\begin{abstract}
We use variational methods and a modified curvature flow to give an alternative proof of the existence of a self-shrinking torus under mean curvature flow. As a consequence of the proof, we establish an upper bound for the weighted energy of our shrinking doughnuts. 
\end{abstract}

\maketitle


\section{Introduction}

Let us consider the half-plane
	$
	\R^2_+ = \{ (r,x): r>0, \, x \in \R \}
	$
equipped with the metric 
	\[
	g =  \lambda^2 g_E = \lambda^2(dr^2 + dx^2), \quad \lambda := r^{n-1} e^{-\frac{1}{4} (x^2 + r^2)}, \quad n\geq 2
	\]
where $g_E$ is the Euclidean metric on $\R^2$. The main result of this article is the following: 
\begin{theorem}
\label{thm:main}
There exists a simple closed geodesic $\gamma_\infty (u)= (r(u), x(u))$, $u \in \S^1$, in the half-plane $(\R^2_+ , g)$.  Moreover, its length $L_n(\gamma_{\infty})$ in the metric $g$ is less than the length of the double cover of the half-line $x=0$: 
	\[
	L_n(\gamma_{\infty})
	< 2 \int_0^\infty s^{n-1} e^{-s^2/4} ds.
	\]
\end{theorem}


When a geodesic such as $\gamma_{\infty}$ is rotated around the $x$-axis, it generates an $n$-dimensional hypersurface $\Sigma$ parametrized by $X: \S^1 \times \S^{n-1} \to \R^{n+1}$:
	\[
	X (u,p) := (r(u)p, x(u)),
	\]
which is self-shrinking under mean curvature flow (see Angenent \cite{angenent;doughnuts}). Self-shrinking hypersurfaces are defined as solutions to the equation $\vec H_{\Sigma} = -\frac{1}{2} X^{\perp}$, where $\vec H$ is the mean curvature vector and $X^{\perp}$ is the projection of $X$ normal to $\Sigma$ and they are of interest because they model the fast-forming singularities of the mean curvature flow \cite{huisken;asymptotic-behavior}. 

Because $\gamma_{\infty}$ is simple and closed, Theorem \ref{thm:main} gives an embedded self-shrinking $\Sigma$ that is a topological torus for all $n\geq 2$. The existence of toroidal self-shrinkers was first proved by Angenent~\cite{angenent;doughnuts} using a shooting method for geodesics (see also Drugan~\cite{drugan;sphere} and Drugan--Kleene \cite{drugan-kleene} for immersed tori). Our proof here uses variational methods and we do not know if our tori coincide with Angenent's shrinking doughnuts.  Indeed, the uniqueness of the shrinking doughnuts is still open in all dimensions $n\geq 2$, 

The idea is to flow simple closed curves in the normal direction with speed $V_g=k_g/K$, where $k_g$ is the curvature and $K$ is the Gauss curvature. This flow was first introduced Poincar\'e \cite{poincare1905} and later studied by Gage \cite{gage1990}. Angenent \cite{angenent;parabolic-equations-curves1} and Oaks \cite{oaks;singularities} considered its generalizations. On manifolds with positive Gauss curvature, such as $(\R^2_+, g)$, the flow decreases lengths. Moreover, if a simple curve encloses a region $\Omega$ for which  the Gauss area $\iint_\Omega K$ is $2\pi$, all evolved curves satisfy the same property thanks to the Gauss-Bonnet formula, and the flow exists for all time. 

Our main contribution is a continuous family of initial rectangles $\Phi(a,0)$, $a \in \R^+$, each one enclosing a region with Gauss area equal to $2 \pi$ and with length less than the double cover of the half-plane $\{ x =0\}$ (and also of the shrinking cylinder $\mathcal C:=\{r=\sqrt{2(n-1)}\}$ by Proposition \ref{prop:Plane<Cylinder}). Because of the maximum principle, rectangles that do not intersect the self-shrinking cylinder $\cyl$ will not intersect it at a later time. For $a$ large, our curves $\Phi(a,0)$ are to the right of $\mathcal C$ and will stay there (in fact, they move further to the right). For $a$ small, our curves are to the left of $\mathcal C$ and get closer to the $x$-axis. Because the family of curves depend continuously on $a$, there is one curve that will not exit to either region and intersect the cylinder at all time. We show that the evolution of this particular curve converges to a geodesic along a subsequence $t_i \to \infty$. 

\begin{figure}[htb]
\begin{picture}(120,90)(0,0)
\linethickness{0.2mm}
\put(10,10){\line(0,1){80}}
\put(10,90){\vector(0,1){0.12}}
\linethickness{0.2mm}
\put(10,50){\line(1,0){110}}
\put(120,50){\vector(1,0){0.12}}

\linethickness{0.2mm}
\multiput(65,10)(0,1.95){40}{\line(0,1){0.90}}
\put(15,85){\makebox(0,0)[cc]{$x$}}

\put(115,45){\makebox(0,0)[cc]{}}

\put(115,45){\makebox(0,0)[cc]{}}

\put(45,75){\makebox(0,0)[cc]{}}

\put(115,45){\makebox(0,0)[cc]{$r$}}

\put(71,32){\makebox(0,0)[cc]{$\mathcal C$}}




%
%
\end{picture}
\caption{The self-shrinking cylinder $\mathcal C$ in $(\R^2_+, g)$}	
\end{figure}
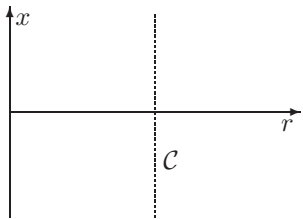

We need only a few simple facts to prove convergence to a geodesic. This is the strength of our approach. The bulk of the work resides in proving that the length of the initial curves are small enough. In lower dimensions, the computation of the lengths can be done numerically. To prove the result for all dimensions, we use an induction argument and long expansions in power series. We converted delicate work on differential equations to estimates of length. What is lost in elegance might be gained in practicality. Indeed, numerics could provide results in lower dimension and different contexts. Even though our proof relies on powerful theorems for the long-time existence of the flow, these theorems can be applied in many situations. 

The article is structured as follows. In Section~\ref{sec:formulas}, we gather useful facts, such as the curvature and length of a curve, the equation for geodesics in $(\R^2_+, g)$, and the Gauss-Bonnet formula. In Section~\ref{sec:csf}, we introduce Gage's modified curve shortening flow and check its short-time existence.  The long-time existence for simple closed initial curves enclosing Gauss area of $2 \pi$ is obtained from a result of Oaks \cite{oaks;singularities} using the evolution of lines $r = C$ as barriers. In Section~\ref{sec:rectangles}, we prove that every rectangle that is symmetric with respect to the $r$-axis and has height $2 c_0$ for a well-chosen $c_0>0$ has length less than the double cover of the half-plane. From the rectangles, we choose a one parameter family of initial curves which enclose a Gauss area of $2 \pi$, then select one $\gamma_0$ whose evolution will intersect the cylinder $\cyl$ at all time. In Section~\ref{sec:converge}, we prove that the flow of $\gamma_0$ converges to a simple closed geodesic along a subsequence of times $t_i \to \infty$. Finally, in the Appendix, we prove some basic facts on the behavior of $\lambda$ that are used in Section~\ref{sec:rectangles}. 

\subsection*{Acknowledgment} The authors would like to thank Stephen Kleene and Sigurd Angenent. Stephen Kleene introduced the authors in the hope to solve a related problem. That problem is still open and morphed into this one. Sigurd Angenent generously shared his expertise on the existence, properties, and convergence of parabolic flows.  


\section{Curvature, geodesics, length, and the Gauss-Bonnet formula}
\label{sec:formulas}

Consider the half-plane
	$
	\R^2_+ = \{ (r,x): r>0, \, x \in \R \}
	$
equipped with the metric 
	\[
	g =  \lambda^2 g_E = \lambda^2(dr^2 + dx^2), \quad \lambda := r^{(n-1)} e^{-\frac{1}{4} (x^2 + r^2)}, \quad n\geq 2
	\]
where $g_E$ is the Euclidean metric on $\R^2$.

\subsection*{Notation.}
The vectors $\partial /\partial r$ and $\partial/\partial x$ form an orthonormal basis in $\R^{2}$ for the usual Euclidean metric; they have length $\lambda$ in the metric $g$. To avoid confusion, we denote with a subscript $g$ geometric quantities and unit vectors taken with respect to the metric $g$ and we use a subscript $E$ when we refer to the Euclidean metric $g_E$.

\subsection*{Curvature and geodesic equation}
Given a curve $\gamma(u) = (r(u), x(u))$ in $\R^2_+$, the speed, unit tangent and normal vectors are given by 
	\begin{gather}
	v = \lambda \sqrt{(x')^2 + (r')^2}, \quad 
	\label{eq:TangentNormal}\mathbf t_g= \frac{1}{v} \left(r' \frac{\pd}{\pd r} + x' \frac{\pd}{\pd x}\right), \quad 
	 \mathbf n_g = \frac{1}{v} \left(-x' \frac{\pd}{\pd r} + r' \frac{\pd}{\pd x}\right).
	\end{gather}
For later computations, we record that $ds = v\ du$ where $s$ is the arclength for $\gamma$ in $(\R^2_+,g)$. The geodesic curvature is 
	\begin{align}
	\label{eq:GeodesicCurvature}
	k_g &= \frac{1}{v} \left[ \frac{x' r'' - x'' r'}{(x')^2 + (r')^2}  - \left( \frac{n-1}{r}  -\frac{r}{2} \right) x' - \frac{1}{2}xr' \right] \\
	& = \frac{1}{v} \left[ k_E  - \left( \frac{n-1}{r}  -\frac{r}{2} \right) x' - \frac{1}{2}xr' \right]. \notag
	\end{align} 
Consequently, the geodesic equation for $(\R^2_+, g)$ is
\begin{equation}
\label{eq:GeodesicEquation}
\frac{x' r'' - x'' r'}{x'^2 + r'^2}  = \left( \frac{n-1}{r}  -\frac{r}{2} \right) x' + \frac{1}{2}xr'.
\end{equation}

\subsection*{Length}
Given a dimension $n \geq 2$, the length of a curve $\gamma(u)$, $a \leq u \leq b$, in $(\R^2_+,g)$ is
\begin{equation}
\label{eq:length}
L_n(\gamma)= \int_a^b v\ du 
\end{equation}
The length will always be taken with respect to $g$, so we drop the subscript. 

\begin{definition}
We use the notation
	\[
	r_n := \sqrt{2(n-1)}.
	\]
The self-shrinking cylinder $\cyl$ corresponds to the geodesic  $(r(u), x(u)) = (r_n, u)$.
The self-shrinking half-line $\pla$ corresponds to the geodesic defined by $(r(u), x(u) ) = (u,0)$.
\end{definition}

The lengths of $\cyl$  and $\pla$ are respectively 
	\begin{gather*}
	L_{n,\cyl} = \int_{-\infty}^{\infty} r_n^{n-1} e^{-\frac{1}{4} (u^2 + r_n^2)} du = 2 \sqrt{\pi}\, r_n^{n-1} e^{-(n-1)}\\
	L_{n,\pla} = \int_0^{\infty} u^{n-1} e^{-u^2/4} du. 
	\end{gather*}

\begin{proposition}
\label{prop:Plane<Cylinder}

For $n \geq 2$, the cylinder is longer than the plane, i.e.
	\begin{equation}
	\label{eq:Plane<Cylinder}
	L_{n,\pla} < L_{n,\cyl}.
	\end{equation}
\end{proposition}
We give here an explicit proof with computations, although the inequality also follows from Huisken's monotonicity formula in \cite{huisken;asymptotic-behavior}. 
\begin{proof}
Using integration by parts for $L_{n,\pla}$, we obtain for $k \in \mathbb N$
	\begin{align*}
		&L_{2k,\pla} = 2^{2k-1}(k-1)!, &  &L_{2k,\cyl} = 2 \sqrt \pi \left(\frac{2 (2k-1)}{e}\right)^{\frac{2k-1}{2}},\\
	&L_{2k+1, \pla} = \sqrt \pi \  2^k\cdot 3 \cdot 5 \cdots (2k-1), & &L_{2k+1,\cyl} = 2 \sqrt \pi \left(\frac{2 (2k)}{e}\right)^{k}.
	\end{align*}
Stirling's approximation from \cite{babyrudin} or \cite{wikipedia_stirling} now gives 
	\begin{equation}
	\label{eq:Stirling}
	\sqrt{2 \pi}\, n^{n+1/2} e^{-n} \leq n! \leq e  n^{n+1/2} e^{-n}.
	\end{equation}
Because  $e < 2 \sqrt \pi$, we have
	\begin{align*}
	L_{2k,\pla}  \leq 2^{2k-1}  (k-1)^{\frac{2k-1}{2} } e^{-k+2}  =  \frac{e^{3/2}}{2\sqrt \pi} \left(\frac{2k-2}{2k-1}\right)^{\frac{2k-1}{2}} L_{2k, \cyl}  < L_{2k, \cyl}.
	\end{align*}
 We estimate the odd dimensions similarly:
	\begin{align*}
	L_{2k+1, \pla}  = \sqrt \pi \frac{(2k)!}{k!} \leq \frac {\sqrt \pi e^{-2k+1} (2k)^{2k+1/2}}{\sqrt{2\pi} e^{-k} k^{k+1/2}}  \leq \frac{e}{2\sqrt \pi} L_{2k+1, \cyl}  < L_{2k+1, \cyl}. 
	\end{align*}
\end{proof}

\subsection*{Gauss curvature and Gauss-Bonnet formula}

The Gauss curvature $K$ of $(\R^2_+, g_{S})$ is given by
	\[
	K = \lambda^{-2}\left(1 +  \frac{n-1}{r^2} \right).
	\]
\begin{definition}
The enclosed {\bf Gauss area} of a closed curve $\gamma: \mathbb S^1 \to \R^2_+$ is the integral
	\begin{equation}
	\label{eq:GADefinition}
	GA_n(\gamma):=\iint_{\Omega} \left( 1 + \frac{n-1}{r^2}\right) dx dr,
	\end{equation}
where $\Omega$ is the region enclosed by  $\gamma$. 
\end{definition}

For a $C^1$ curve $\gamma$, the Gauss-Bonnet formula may be written as follows
\begin{equation}
\label{eq:GaussBonnet}
 \iint_{\Omega} \left( 1 + \frac{n-1}{r^2}\right) dx dr = 2 \pi - \oint_{\gamma} k_g ds.
\end{equation}


\section{A modified curve shortening flow}
\label{sec:csf}

Given a closed curve in $\R^2_+$, we consider the flow that moves the curve according to the following normal velocity
\begin{equation}
\label{eq:NormalVelocity}
V_g = \frac{k_g}{K},
\end{equation}
where $k_g$ is the geodesic curvature and $K$ is the Gauss curvature at the given point on the curve. When $K$ is positive and uniformly bounded from zero, this weighted flow exhibits properties similar to the usual curve shortening flow. 

The flow was studied by Gage \cite{gage1990} to show the existence of geodesics on spheres via variational methods. The idea is similar here, but our manifold is not closed and the Gauss curvature is not bounded. We will use lines $r=r(t)$ as barriers to show that closed curves in the interior of $\R^2_+$ do not reach regions where the Gauss curvature blows up in finite time.

\subsection*{Evolution of lines $r= C$.}

From  \eqref{eq:GeodesicCurvature}, the lines $r=r_0$ have constant euclidean velocity at all points. Indeed the curvature is 
	\[
	k_g =  \lambda^{-1}  \left( \frac{r_0}{2}  -\frac{n-1}{r_0}  \right)
	\]
The Euclidean speed $V_E$ in the direction $\partial/\partial x$ is  
	\[
	 V_E =\frac{ V_g}{\lambda} =\frac{k_g}{ \lambda K}= \frac{r( r^2 - 2(n-1))}{2 (r^2 + n-1)}.
	\]
Lines on the left side of the cylinder $\cyl$, i.e. $r = r_0 < r_n$, move further to the left. Similarly, lines on the right side of the cylinder $\cyl$ move to the right.  When $r_0 < \rncyl$, we have $V_E \geq -cr $  with $c$ a positive constant depending only on $r_0$ and $n$. Therefore 
	\[
	r(t) \geq r_0 e^{-ct}, \quad r_0 < r_n,
	\]
so no vertical line reaches the $x$-axis in finite time. Similar,  when $r_0 > \rncyl$, we have $r(t) \leq r_0 e^{t/2}$ so no line goes to infinity in finite time either. 

\subsection*{Short-time and long-time existence.}
In the case where the positive Gaussian curvature is uniformly bounded above and below away from zero, the short-time existence of the flow is given by Gage \cite{gage1990} or Angenent \cite{angenent;parabolic-equations-curves1}.  Given an initial embedded closed curve $\gamma_0: \mathbb S^1 \to \R^2_+$, we choose $r_0 <r_n$ and $r_1>r_n$ so that $\gamma_0(\S^1)$ is within the slab $\{ r \in (r_0, r_1)\}$. For $t<1$, the Gauss curvature is then uniformly bounded in $\{ r \in (r_0 e^{-c}, r_1 e^{1/2})\}$ and we can apply the short-time existence results there. 

Some properties of $\gamma_0$ are preserved under the flow: 
\begin{proposition} 
\label{prop:SymmetryGraphical}
Let us denote by  $\gamma_t$ the evolution of $\gamma_0$ at time $t$. 
\begin{enumerate}
\item If $\gamma_0$ is an embedded curve, so is $\gamma_t$. 
\item  If $\gamma_0$ is symmetric with respect to reflections across the $r$-axis and the image of $\gamma_0$ in the first quadrant is a graph over the $r$-axis, the same properties (symmetric and graphical) hold for all the $\gamma_t$ as long as the flow exists. 
	\end{enumerate}
	\end{proposition}
	\begin{proof}
Both items are direct consequences of Theorem 1.3 \cite{angenent;parabolic-equations-curves2}, which states that the number of intersection points of two curves (or self-intersections of one curve) is non-increasing. For (2), symmetries are preserved because of the uniqueness of the flow. The initial curve $\gamma_0$ intersects each vertical line at most twice and vertical lines evolve into other vertical lines. Symmetry and the vertical line test now imply that the piece of $\gamma_t$ in the first quadrant is graphical over the $r$-axis. 
	\end{proof}

The arc length $ds = \lambda \sqrt{(dx)^2+(dr)^2}$ evolves according to
\[
\frac{\pd}{\pd t} ds = - k_gV_g \, ds = - \frac{k_g^2}{K} \, ds. 
\]
This implies that the length $L_n(\gamma_t)$ evolves by 
\begin{equation}
\label{eq:EvolutionLength}
\frac{d}{dt} L_n(\gamma_t) = -\int_{\gamma_t} \frac{k_g^2}{K}\, ds.
\end{equation}
If the $\gamma_t$'s are simple and closed then they are boundaries of domains $\Omega_t$. Using the Gauss-Bonnet formula \eqref{eq:GaussBonnet}, we get  
%
\begin{align*}
\frac{d}{dt} \iint_{\Omega_t} KdA 
= -\oint_{\gamma_t} K\, V_g\, ds 
= -\oint_{\gamma_t} k_g\,  ds 
= -2\pi + \iint_{\Omega_t} KdA.
\end{align*}
Here $V_g$ and $k_g$ are the normal velocity and geodesic curvature in the direction of the \emph{inward} normal to $\Omega_t$. We have the following proposition. 
	\begin{proposition}
	\label{prop:GaussArea}
	If the Gauss area enclosed by the initial curve is equal to $2 \pi$, then the Gauss area enclosed by $\gamma_t$ is also $2 \pi$ as long as the flow exists. 
	\end{proposition}
	
We are now ready to prove long-time existence.
\begin{proposition}[Long-time existence]
Let $\gamma_0$ be a simple closed curve. If the domain $\Omega_0$ enclosed by $\gamma_0$ satisfies 
	$
	\iint_{\Omega_0} K dA = 2 \pi
	$
then the evolution of $\gamma_0$ with normal velocity $V_g = k_g/K$ exists for all time.  
\end{proposition}
\begin{proof}
Oaks proved that a simple closed curve either shrinks to a point in finite time or exists for all time if the velocity $V$ satisfies  $\lambda ^{-1}\leq \pd V / \pd k \leq \lambda$ for a constant $\lambda >0$ \cite[Corollary 6.2]{oaks;singularities}. In our situation, the condition is equivalent to uniform bounds from above and from below away from zero for the Gauss curvature. 

Let us assume for the sake of contradiction that $\gamma_t$ develops a singularity at a finite time $T$. Using vertical lines as barriers as in our argument for short-time existence, we know that $\gamma_t$, $t \in (0,T)$ stays within a slab $\{ r_0 e^{-cT} \leq r \leq r_1 e^{T/2} \}$, where the Gauss curvature is uniformly bounded. Yet $\gamma_0$ develops a singularity under the flow in finite time, which contradicts the result of Oaks.
\end{proof}


\section{A family of initial curves}
\label{sec:rectangles}

We consider rectangles $R[a,b,c]$ with vertices $(a, -c), (a, c), (b,c), (b,-c)$, $a<b$ and $c>0$.

\begin{figure}[htb]
\begin{picture}(120,90)(0,0)
\linethickness{0.2mm}
\put(10,10){\line(0,1){80}}
\put(10,90){\vector(0,1){0.12}}
\linethickness{0.25mm}
\put(10,50){\line(1,0){110}}
\put(120,50){\vector(1,0){0.12}}
\linethickness{0.25mm}
\multiput(40,30)(0,1.95){21}{\line(0,1){0.98}}
\linethickness{0.2mm}
\multiput(40,30)(1.96,0){26}{\line(1,0){0.98}}
\linethickness{0.2mm}
\multiput(90,30)(0,1.95){21}{\line(0,1){0.98}}
\linethickness{0.2mm}
\multiput(40,70)(1.96,0){26}{\line(1,0){0.98}}
\put(15,85){\makebox(0,0)[cc]{$x$}}

\put(115,45){\makebox(0,0)[cc]{}}

\put(115,45){\makebox(0,0)[cc]{}}

\put(45,75){\makebox(0,0)[cc]{}}

\put(115,45){\makebox(0,0)[cc]{$r$}}

\put(35,45){\makebox(0,0)[cc]{$a$}}

\put(95,45){\makebox(0,0)[cc]{$b$}}

\put(3,75){\makebox(0,0)[cc]{$c$}}

\put(1,35){\makebox(0,0)[cc]{$-c$}}

\put(10,70){\makebox(0,0)[cc]{}}

\linethickness{0.2mm}
\linethickness{0.2mm}
\linethickness{0.2mm}
\put(8,70){\line(1,0){4}}
\linethickness{0.2mm}
\put(8,30){\line(1,0){4}}
\end{picture}
\caption{The rectangle $R[a,b,c]$.} 
\end{figure}
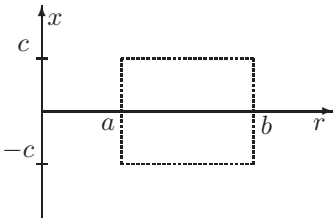
	
If our initial rectangle encloses Gauss area equal to $2 \pi$, the flow will exist for all time. If in addition, its perimeter is less than $2 L_{n, \pla}$, the flow can not converge to a double cover of a plane because it decreases length. There are infinitely many rectangles satisfying both of these conditions.  For example, when $a$ and $b$ are large, the Gauss area $\iint 1 + \frac{n-1}{r^2} dx dr \sim \iint dx dr $ but the perimeter will be small because of the exponential weight on the metric. When $a$ and $b$ are close to $0$, the integral $\int_0^1 \frac{1}{r^2} dr$ is unbounded, so it is possible to have tiny rectangles with Gauss area $2 \pi$ and very small perimeter also. Here, we choose a continuous one parameter family that bridges these two extremes. 

\begin{definition}

We denote by $L(a,b,c)$ the perimeter of the rectangle $R[a,b,c]$:
	\begin{equation}
	\label{eq:PerimeterRectangle}
	L_n(a,b,c) := 2 (a^{n-1} e^{-a^2/4} + b^{n-1} e^{-b^2/4}) \int_0^c e^{-x^2/4}dx + 2 e^{-c^2/4} \int_a^b r^{n-1} e^{-r^2/4} dr 
	\end{equation}
Fix $c_0$ to be the positive real number such that 
	\[
	\frac{e^{-c_0^2/4}}{\int_0^{c_0} e^{-x^2/4}dx}=2.
	\]
\end{definition}
%

The number $c_0$ exists and is unique. It is approximately $0.481$. The number is small enough so that, when $a$ is close to $b$, $L_n(a,b,c_0)$ stays away from $2 L_{n,\pla}$ and yet  large enough so that when $a \to 0$ and $b \to \infty$, there is some room between the perimeter of the rectangle $R[a,b,c_0]$ and $2 L_{n,\pla}$. Our first guess, informed by numerics, was $c=1/2$ but the computations are cleaner with $c_0$. 

\begin{proposition}
\label{prop:PerimeterRectangle}
For every $a,b \in (0, \infty)$ and $n\geq 2$, we have
	\begin{equation}
	\label{eq:PerimeterBounded}
	L_n(a,b,c_0) < 2 L_{n,\pla}.
	\end{equation}
\end{proposition}
The proof uses numerical inspection for a few lower dimensions then induction on the dimension from $n$ to $n+2$.
\begin{proof}
We write length of the rectangle $R[a,b,c]$ given in \eqref{eq:PerimeterRectangle}
%
%
%
as $L_n(a,b,c_0) = f_n(a) + g_n(b)$, where $M= \int_{0}^{c_0} e^{-x^2/4}dx$ and
	\begin{align*}
	f_{n}(a) &=  2 a^{(n-1)} e^{-a^2/4} M + 4M \int_a^{0} r^{n-1} e^{-r^2/4} dr \\
	g_{n}(b) &= 2 b^{(n-1)} e^{-b^2/4}  M + 4M \int_{0}^b r^{n-1} e^{-r^2/4} dr.
	\end{align*}
We maximize $f$ and $g$ separately. \\

\noindent {\bf Locating the $\max f$ and $\max g$.} The derivative of $f$ is 
	\[
	f_n'(a) =  -a^{(n-2)} e^{-a^2/4} M  \left(  a^2 +4 a - 2 (n-1) \right).
	\]
So $\max_{a \in [0,\infty)} f_n(a)$ is achieved at $a_n = -2 + \sqrt{2(n+1)}$. Similarly, one can show that $\max_{b \in [0,\infty)} g_n(b)$ is achieved at $b_n = 2 + \sqrt{2(n+1)}$.\\

\noindent  {\bf Base cases: The estimate \eqref{eq:PerimeterBounded} is true for $n=2$, $n=3$ and $n=4$.} We simply verify the estimate \eqref{eq:PerimeterBounded} numerically. 


\begin{claim}[Induction step]
\label{claim:InductionStep}
 	\begin{equation}
	\label{eq:InductionStep}
 	f_{n+2}(a_{n+2}) + g_{n+2}(b_{n+2}) < 2n(f_n(a_n)+g_n(b_n)), \quad n\geq 3
 	\end{equation}
\end{claim}	
\begin{proof}[Proof of the induction step]
By integration by parts, $2n \int_a^b r^{n-1} e^{-r^2/4} dr =  2b^n e^{-b^2/4} - 2a^n e^{-a^2/4} + \int_a^b r^{n+1} e^{-r^2/4} dr$, therefore for all $a, b >0$,
	\begin{gather*}
	2n f_n(a)-2f_{n+1}'(a)- f_{n+2}(a)=0, \\
	2ng_n(b) -2g'_{n+1}(b) - g_{n+2}(b)=0.
	\end{gather*}
Because $f'_{n+1} (a_{n+1})=g'_{n+1}(b_{n+1})=0$, the estimate \eqref{eq:InductionStep} is equivalent to 
	\[
	\int_{a_{n+1}}^{a_{n+2}} f'_{n+2}(u) du + \int_{b_{n+1}}^{b_{n+2}} g'_{n+2}(v) dv < 2n \left(\int_{a_{n+1}}^{a_{n}} f'_{n}(u) du+\int_{b_{n+1}}^{b_{n}} g'_{n}(v) dv \right).
	\]
Using the variables $u=s-2$ and $v=s+2$ and combining like terms, we obtain one more equivalent form:
	\begin{multline}
	\label{eq:LastIntegral2}
	\int_{r_{n+3}}^{r_{n+4}} (2(n+3)-s^2) [e^{-(s-2)^2/4} (s-2)^n + e^{-(s+2)^2/4} (s+2)^n ] ds \\
		<  2n \int_{r_{n+2}}^{r_{n+3}}   (y^2-2(n+1)) [e^{-(y-2)^2/4}(y-2)^{n-2} + e^{-(y+2)^2/4}(y+2)^{n-2} ]dy
	\end{multline}
For $  n \in \mathbb N$, we define the functions
	\[
	h_n (s) =  e^{-s^2/4} s^n, \quad H_n(s) = h_n(s-2) + h_n(s+2),
	\]
 and substitute $s=r_{n+4} -t$ and $y=r_{n+2}+t$ to get yet another equivalent form:
	\begin{multline*}
	\int_0^{r_{n+4}-r_{n+3}} t  (2 r_{n+4} - t) H_n(r_{n+4}-t)dt \\
	< 
	 \int_0^{r_{n+3}-r_{n+2}} 2n\  t ( 2 r_{n+2}+t)  H_{n-2}(r_{n+2}+t) dt.
	\end{multline*} 
Because of concavity, we have $r_{n+4} - r_{n+3} < r_{n+3} - r_{n+2}$. Therefore, in order to prove \eqref{eq:InductionStep}-\eqref{eq:LastIntegral2}, it suffices to compare the integrands pointwise and show
	\begin{equation}
	\label{eq:pointwise}
	 (2 r_{n+4} - t) H_n(r_{n+4}-t) < 2n ( 2 r_{n+2}+t)  H_{n-2}(r_{n+2}+t)
	\end{equation}
 for $0<t<r_{n+4} -r_{n+3}$. We claim that the inequality \eqref{eq:pointwise}
is at its tightest when $t=0$. Indeed, for $n \geq 7$ and $s \in (r_{n+3}, r_{n+5})$, $H_n''(s)$ is positive so $H_n(s)$ is increasing because $H_n'(r_{n+3})>0$ (see Appendix for more detail). For $n \leq 6$, $e^{(s^2+4)/4} H'_n(s)$ is positive by numerical inspection.


The rest of this section is dedicated to proving the following equivalent formulation of the estimate \eqref{eq:pointwise} at $t=0$:
	\begin{equation}
	\label{eq:t=0}
	\ln\left(\mathcal H(r_{n+4})\right) < \ln\left( e \frac{n (n+3)}{(n+1)^2} \left(\frac{n+1}{n+3}\right)^{\frac{n+3}{2}}\mathcal H(r_{n+2})\right) =: \text{I} + \ln(\mathcal H(r_{n+2})),
	\end{equation}
where for $m =  m(s) = \frac{s^2}{2}-3$, 
	$
	\mathcal H(s) :=s^{-m} H_m(s) e^{(s^2+4)/4}= \left(1 - \frac{2}{s}\right)^m e^s + \left(1 + \frac{2}{s}\right)^m e^{-s}.
	$
Proving \eqref{eq:t=0} is a laborious exercise of expanding $\mathcal H$ in series for $s$ large. 	

We recall the power expansion of $\ln$ for $1>y>0$ and use comparison to a geometric series with ratio $y$ to estimate the remainder in \eqref{eq:ln(1-y)}.
	\begin{align}
	\label{eq:ln(1+y)}
	\ln(1+y) &= y -\frac{y^2}{2} + \cdots + (-1)^{j+1} \frac{y^j}{j} + R^+_j(y), &
	 -\frac{y^{2K+2}}{2K+2} \leq R^+_{2K+1}(y) \leq 0, \\
	 	\label{eq:ln(1-y)} 
	\ln(1-y) &= - y - \frac{y^2}{2} - \cdots - \frac{y^ j}{j} + R^-_j(y), &
	- \frac{y^{j+1}}{(j+1)(1-y)} \leq R^-_{j}(y) \leq 0.
	\end{align}
Separating the even and odd powers, we have
	\begin{align*}
\left(\frac{s^2}{2} -3\right) \ln \left(1 -\frac{2}{s} \right) + s 	& = E_{2K-2}(s) + O_{2K-1}(s) + \mathcal R^{-}_{2K-1}(s),\\
	\left(\frac{s^2}{2} -3\right) \ln \left(1 +\frac{2}{s} \right) - s 	& =E_{2K-2}(s) - O_{2K-1}(s)+ \mathcal R^{+}_{2K-1}(s).
	\end{align*}
For the odd powers, we just remark that the function $O_{2K-1}$ is decreasing in $s$. We will need the formulas  for the even powers and the remainder
	\begin{align*}
	E_{2K-2}(s) := &-1+ \sum_{j=1}^{K-1}    \left( \frac{3}{2j} - \frac{2}{2j+2}\right) \left(\frac{2}{s} \right)^{2j} \\
	\mathcal R^{\pm}_{2K-1}(s) :=& - 3 R^{\pm}_{2K-1}(\tfrac{2}{s}) + \frac{s^2}{2} R^{\pm}_{2K+1}(\tfrac{2}{s}).
	\end{align*}
From \eqref{eq:ln(1+y)} and \eqref{eq:ln(1-y)}, we note the (lower) bound for $R_j^-$ dominates, therefore it suffices to consider $\mathcal R^-$. We obtain
	\begin{gather*}
	- \frac{{2}}{(K+1)(s-2)} \left(\frac{2}{s}\right)^{2K-1}	\leq \mathcal R^{-}_{2K-1}(s) \leq   \frac{3}{K(s-2)}\left(\frac{2}{s}\right)^{2K-1}.
	\end{gather*}
Therefore, for $m(s)={\frac{s^2}{2}-3}$, 
	\begin{align*}
		\mathcal H(s) &> 2 \exp\left((E_{2K-2}(s)    - \frac{{2}}{(K+1)(s-2)} \left(\frac{2}{s}\right)^{2K-1}\right) \cosh(O_{2K-1}(s) ),\\
		\mathcal H(s) &<  2 \exp\left(E_{2K-2}(s) + \frac{3}{K (s-2)}\left(\frac{2}{s}\right)^{2K-1}\right) \cosh(O_{2K-1}(s)).
	\end{align*}
Because $\cosh(O_{2K-1}(s)) > \cosh(O_{2K-1}(s +\ve))$, to finish proving \eqref{eq:t=0}, we show that 
	\begin{multline}
	\label{eq:compareE}
	\left(E_{2K-2}(s) - E_{2K-2}(s+\ve)   - \left(\frac{2}{(K+1)} + \frac{3}{K}\right) \frac{1}{(s-2)} \left(\frac{2}{s}\right)^{2K-1} \right)	+ \text{I} >0
	\end{multline}
for $s=r_{n+2}$ and $s+\ve = r_{n+4}$ and where I was defined implicitly in \eqref{eq:t=0}. The values for $s$ and $s+\ve$ are now fixed.

First, we estimate the last term of \eqref{eq:compareE} by expanding in power series and keeping lower order terms in $(n+1)^{-1}$ and $(n+3)^{-1}$. 
	\begin{align*}
	  \text{I} 
	  & >  1 + \left( \frac{1}{n+1} - \frac{2}{(n+1)^2}\right) - \left( \frac{1}{n+1} - \frac{2}{(n+1)^2}\right)^2 \\
	  	& \ \ \ \ \ \ -1  -  \frac{1}{n+3} - \frac{1}{3} \left( \frac{2}{n+3}\right)^2 - \frac{1}{4} \left( \frac{2}{n+3}\right)^3 \frac{(n+3)}{(n+1)}\\
	 & >   \frac{2}{(n+1)(n+3)} - \frac{3}{(n+1)^2} - \frac{4}{3 (n+3)^2} \\ 
	 & >  -\frac{3}{(n+1)(n+3)} - \frac{15}{(n+1)^2(n+3)^2}.
	\end{align*}
The terms of lower order of $E_{2K-2}(s) - E_{2K-2}(s+\ve)$ are explicitly 
	\begin{align*}
	E_{2K-2}(s) - E_{2K-2}(s+\ve) & >\frac{4}{2(n+1)} - \frac{4}{2(n+3)} + \frac{20}{3(2(n+1))^2} - \frac{20}{3(2(n+3))^2} \\
		& = \frac{4}{(n+1)(n+3)} + \frac{20(n+2)}{3(n+1)^2(n+3)^2}.
	\end{align*}
Because $\frac{4}{s} > \frac{1}{(s-2)}$ for $s>\sqrt8$ (i. e. when $n\geq 3$), the error term is  estimated by 
	\[
	- \left(\frac{2}{(K+1)} + \frac{3}{K}\right) \frac{1}{(s-2)} \left(\frac{2}{s}\right)^{2K-1} > -10 K^{-1} (2/s)^{2K}.
	\]
Recalling that $s = r_{n+2}$, we obtain 
	\begin{multline*}
	E_{2K-2}(s) - E_{2K-2}(s+\ve) + I  -  10 \frac{2^{2K} }{K s^{2K}}
		>
		\frac{1}{(n+1)(n+3)}  -  10 \frac{2^{K} }{K (n+1)^{K}},
	\end{multline*}
which is positive for $K$ large enough. This concludes the proof of Claim \ref{claim:InductionStep}.
\end{proof}
This also concludes the proof of Proposition \ref{prop:PerimeterRectangle}.
\end{proof}

For a fixed dimension $n$, we round the corners of the rectangles while keeping the length less then $2 L_{n,\pla}$. This step is not necessary as one still has short-time and long-time existence for Lipschitz initial conditions. It just simplifies the exposition. 
\begin{definition}
Let us define $\mathcal R[a,b,c_0]$ as the rectangle $R[a,b,c_0]$ with rounded corners. Given $n\geq 2$, all the rectangles are rounded off in the same manner and such that $L_n(\mathcal R[a,b,c_0]) \leq 2 L_{n,\pla}$ for all $a,b>0$. 
\end{definition}  

As an immediate corollary, we can extract a one parameter family of  rounded  rectangles with Gauss area $2 \pi$. 
\begin{corollary}
There is a smooth function $\varphi: \R_+ \to \R_+$ with $\varphi(a)>a$ so that  the family of rectangles $\mathcal R[a,\varphi(a),c_0]$ satisfies 
	\[
	GA_n(\mathcal R[a,\varphi(a),c_0]) = 2\pi, \quad L(a,\varphi(a),c_0) < 2 L_{n,\pla}.
	\]
\end{corollary}
Note that $\lim_{a\to 0} \varphi(a) = 0$.

\begin{proposition}  Let $\Phi: \R_+ \times \R_+ \to C^{0} (\mathbb S^1, \R^2_+)$ be the map with the following properties:
	\begin{enumerate}
	\item $\Phi(a, 0) = \mathcal R[a,\varphi(a), c_0]$,
	\item For fixed $a$, $\Phi(a, t)$ satisfies the evolution equation \eqref{eq:NormalVelocity}.
	\end{enumerate}
There exists an $a_0 \in \R_+$ so that $\Phi(a_0,t)$ intersects the cylinder for all time $t \in \R_+$. 
\end{proposition}	

\begin{proof} The set of curves that do not intersect the cylinder is split into two disjoint sets
 	\begin{align*}
	A_1 & = \{ \text{continuous closed curves } \gamma: \S^1 \to \R^2_+ \mid \gamma(s) < r_n,  s \in S^1\} ,\\
	A_2 & = \{ \text{continuous closed curves } \gamma: \S^1 \to \R^2_+ \mid \gamma(s) > r_n,  s \in S^1\}. 
	\end{align*}
The cylinder $\cyl$ is a self-shrinker so it is stationary under our flow. By the maximum principle, if $\Phi(a,t_0) \in A_i$ for some $i=1,2$ and some time $t_0$, then $\Phi(a, t) \in A_i$ for $t \geq t_0$.
We now consider the following subsets of $\R_+$:
	\begin{align*}
	U_i & = \{ a \in \R_+ \mid \exists t >0, \Phi(a,t) \in A_i\}
	\end{align*}
Both $U_1$ and $U_2$ are open and $U_1 \cap U_2 = \emptyset$. Therefore $U_1 \cup U_2 \neq \R_+$ and we have our claim.
\end{proof}


\section{Convergence to a geodesic}
\label{sec:converge}

We are left to prove that the curves $\gamma_t:= \Phi(a_0,t)$ converge to a shrinking doughnut along a subsequence. This is done in three steps. First, we show that on compact sets, the curves $\gamma_t$ approach a geodesic along a subsequence. Secondly, we recall that each curve $\gamma_t$ restricted to the first quadrant is a graph over $r$ so we study graphical geodesics. Finally, from the behavior of nearby geodesics, we show that the subsequence of curves $\gamma_t$ stay within a compact domain. In this section, the dimension $n$ is fixed. 


\subsection{Convergence to geodesics in compact sets}
We first extract a sequence with decaying total curvature. 

\begin{proposition}
\label{prop:sequenceti} There is a sequence $t_{i}$ so that 
	\begin{equation}
	\label{eq:TotalCurvatureDecay}
	\int_{\gamma_{t_i}\cap E}  |k|\, ds \to 0
	\end{equation}
for any compact subset $E$ of the open half-plane $\mathbb R^2_+$. 
\end{proposition}

\begin{proof}
From the evolution of the length \eqref{eq:EvolutionLength} and the fact that  $k^2/K$ is positive, we have
 	\[
	\int_0^{\infty}\!\! \!\int_{\gamma_t} \frac{k^2}{K}\, ds \, dt \leq L_n(\gamma_0) < \infty
	\]
so that 
	$
	\int_i^{i+1}\!\! \! \int_{\gamma_t} \frac{k^2}{K}\, ds \, dt \to 0
	$
as $i\to \infty$. Therefore there exists a sequence $\{ t_i\}$ with $t_i \in [i, i+1]$ so that 
	\[
	\int_{\gamma_{t_i}} \frac{k^2}{K}\, ds \to 0.
	\]
Because $E$ is compact, the Gauss curvature is uniformly bounded in $E$ so
	\[
	\left(\int_{\gamma_{t_i} \cap E} |k|\, ds\right)^2 \leq L_n(\gamma_{t_i} \cap E) \int_{\gamma_{t_i} \cap E} k^2 ds 
	\]
and we have \eqref{eq:TotalCurvatureDecay} because the length is bounded. 
\end{proof}
From now on, let us use the notation $\gamma_i := \gamma_{t_i}$. 
If $\gamma_{i}$ visits a compact set $E$ frequently then we can extract a subsequence of points $\gamma_{i}(x_i)$ that converges in $E$. The following proposition shows we have $C^1$ convergence of the pieces of curves containing these points.  

\begin{proposition}
\label{prop:C1Convergence} Let $t_i$ be the sequence from Proposition \ref{prop:sequenceti} (or possibly one of its subsequences) and let $E$ be a compact set in $\mathbb R^2_+$. If for some $x_i \in \S^1$, $i \in \mathbb N$,  the sequence $\{\gamma_{i} (x_i)\}$ converges to a point $P \in E$, then there exists a subsequence ${i_j}$ so that the connected component of $\gamma_{i_j} \cap E$ containing $\gamma_{i_j} (x_{i_j}) $ converges in $C^1$ in $E$.
The limit curve contains $P$ and satisfies the geodesic equation in $E$. 
\end{proposition}

Note that on compact sets, the metric $g$ is equivalent to the standard Euclidean metric, so the notion of $C^1$ convergence is the same regardless of the metric. We will use the standard metric. 
\begin{proof}
We first prove convergence in $C^0$ in a slightly bigger set. Let $E'$ be a compact set of $\R^2_+$ that contains a neighborhood of $E$.  Without loss of generality, we can assume that all the curves are parametrized by arclength with respect to $g_E$ and that $\gamma_{t_i}(0)$ converges to $P$. The lengths with respect to $g_E$ of $\gamma_{{i}} \cap E$ are uniformly bounded because $L_n(\gamma_{{i}})$ are bounded. The arcs $\gamma_{{i}} \cap E'$ are then equicontinuous and a subsequence, also denoted by $\gamma_i$ for simplicity, converges uniformly to a continuous arc $\gamma_{\infty}:[c,d] \to E'$. 

We define 
	\[
	 M = \max_{(x,r) \in E'} \left(\left|\frac{n-1}{r} - \frac{r}{2}\right|, \frac{x}{2}\right), \quad \Lambda = \max_{E'} \lambda.
	 \] 
Let $\mathbf t^i_E(u) = \lambda \mathbf t^i_g(u)$  be the unit tangent vector to the curve $\gamma_{i}$ at $u$ and $\theta^i(u)$ the angle $\mathbf t^i_E(u)$ makes with $\partial/\partial r$. 

The point $P$ can now be any point on $\gamma_{\infty} \cap E$. Because of symmetry, we restrict ourselves to the first quadrant. We may reverse the parametrization so that $\gamma_{i}(u)$ is in the first quadrant for all $u \in [-\mu,0]$ for $i$ large enough and $\mu>0$ small enough. Such a constant $\mu$ exists independently of $i$ because all the curves intersect $\cyl$ and enclose a Gauss area of $2 \pi$.  Moreover, we can assume that $\mu$ is smaller than the Euclidean distance between $\pd E'$ and $\pd E$.

We now prove $C^1$ convergence by showing that $\theta^i(0)$ is a Cauchy subsequence.   Suppose $\ve>0$. We pick $\delta <\ve$ so that 
	\begin{enumerate}
	\item $|\sin(\alpha) - \sin(\beta)| \leq \delta \max(\mu^{-1}, 1)$ implies $|\alpha - \beta|  \leq \ve$ if either $\alpha, \beta \in [-\pi/2, \pi/2]$ or $\alpha, \beta \in [\pi/2, 3\pi/2]$,  
	\item$|(r,x) - (\bar r, \bar x)| \leq \delta$ implies $\left|\frac{n-1}{r} - \frac{r}{2}-  \frac{n-1}{\bar r} +\frac{\bar r}{2}\right| \leq \ve$.
	\end{enumerate} 
%
Because of the uniform convergence of $\gamma_{i}$, we have
	\[
	\int_{-\mu}^0 \mathbf t^i_E - \mathbf t^j_E\, du \leq |\gamma_i(a) - \gamma_j (a)| + |\gamma_i(0) - \gamma_j (0)| \leq 2 \delta
	\]
for $i,j$ large enough. 
Restricting to the second component, we have 
	\[
	\frac{1}{\mu} \left| \int_{-\mu}^0 \sin (\theta^i(u)) - \sin (\theta^j(u))\, du \right| \leq  \delta /\mu 
	\]
and by the mean value theorem, there exists a $u^*\in (-\mu,0)$ so that 
	$
	| \sin( \theta^i(u^*) )- \sin (\theta^j(u^*) )| \leq   \delta/\mu
	$
and, from our choice of $\delta$,
	\[
	|\theta^i(u^*) - \theta^j(u^*)| \leq \ve.
	\]
 By equation \eqref{eq:GeodesicCurvature}, we get
	\begin{equation}
	\label{eq:ODEtheta}
	\frac{d\theta^i}{du} = \lambda k^i_g + \left(\frac{n-1}{r^i} - \frac{r^i}{2}\right) \sin\theta^i - \frac{x^i}{2} \cos \theta^i
	\end{equation}
and for $u \in [u^*, 0]$, 
	\begin{align*}
	\left|\theta^i(u) -\theta^j(u)\right| & \leq \left|\theta^i(u^*) -\theta^j(u^*)\right|  + \left|\int_{u^*}^u  (\theta^i)'- (\theta^j)'\, dy\right|   \\
	& \leq 3 \ve + \int_{u^*}^u |\lambda k^i_g(y) | dy+ \int_{u^*}^u |\lambda k^j_g(y) |dy  \\
	&\ \ +\int_{u^*}^u M \big(|\sin(\theta^i(y)) - \sin (\theta^j(y))| + |\cos (\theta^i(y)) - \cos(\theta^j(y)) |\big)\, dy \\
	& \leq  5 \ve +  2 M \int_{u^*}^u |\theta^i(y) - \theta^j(y) |\, dy 
	\end{align*}
The integral form of Gronwall's inequality gives 
	\begin{equation}
	\label{eq:theta}
	|\theta^i (u) -\theta^j(u)| \leq 5 \ve \left(1 +  e^{2M  (u-u^*)}\right) = 5 \ve C 
	\end{equation}
where $C$ depends on $E'$ and $\mu$ only. The sequence of angles $\{ \theta^i\}$ at $P$ is therefore a Cauchy sequence. Because the point $P$ is arbitrary, we have convergence in $C^1$. 

The limit curve is a geodesic. If not, there would exists $u$ and $v$ so that the limit function $\theta^{\infty}$ would satisfy
	\[
	\theta^{\infty}(u) - \theta^{\infty}(v) - \int_u^v h(r^{\infty}(y), x^{\infty}(y), \theta^{\infty}(y)) dy = \ve \neq 0,
	\]
for some $\ve$ and where $h(r,x,\theta) =  \left(\frac{n-1}{r} - \frac{r}{2}\right) \sin\theta - \frac{x}{2} \cos \theta$. 
The $C^1$ convergence combined with \eqref{eq:ODEtheta} would give
	\[
	\int |k_g| ds \geq \left |\theta^{i}(u) - \theta^{i}(v) - \int_u^v h(r^{i}(y), x^{i}(y), \theta^{i}(y)) dy\right| \geq \ve/2
	\]
for all $i$ large enough, in contradiction with  \eqref{eq:TotalCurvatureDecay}. The curve $\gamma_{\infty}$ is therefore a (piece of) geodesic. 
\end{proof}

\subsection{Properties of positive solutions to the graphical geodesic equation} 

When $\gamma = (r, f(r))$ is a graph over the $r$-axis, the geodesic equation is equivalent to 
\begin{equation}
\label{eq:GeodesicEquationf}
\frac{f''}{1+f'^2} =  \left(\frac{r}{2}  - \frac{n-1}{r} \right) f' - \frac{1}{2} f.
\end{equation}



\begin{lemma}
	\label{lem:small_facts_geodesics}
	A positive solution  $f:I \to \R$ to  \eqref{eq:GeodesicEquationf} has the following properties:
	\begin{enumerate}
	\item the function $f$  does not have a local minimum in $I$;
	\item if $f'(\rho_0) \geq 0$ for some $\rho_0 \in (0, r_n)\cap I$, then $f''<0$ in $(\rho_0, r_n) \cap I$; 
	\item if $f'(\rho_0) \leq 0$ for some $\rho_0 \in (r_n,\infty)\cap I$, then $f''<0$ in $(r_n, \rho_0)\cap I$. 
	\end{enumerate}
\end{lemma}

\begin{proof}
Part(1): By examination of equation \eqref{eq:GeodesicEquationf} we see that $f''$ and $f$ have opposite signs when $f'=0$, and thus $f$ cannot have a positive local minimum.

Parts(2,3): First, notice that equation \eqref{eq:GeodesicEquationf}, along with the positivity of $f$, shows that $f''(r)<0$ whenever $r \in (0,r_n) \cap I$ and $f'(r) \geq 0$. In addition, if $f'(\bar{r}) = 0$ at some $\bar{r} \in I$, then $f'(r)<0$ and $f''(r)<0$ for $r$ close to and greater than $\bar{r}$. Finally, differentiating equation \eqref{eq:GeodesicEquationf} gives the equation for the third derivative
		\[
		\frac{f'''}{1+f'^2} = \frac{2 f' (f'')^2}{(1+f'^2)^2} + \left(\frac{r}{2} - \frac{(n-1)}{r}\right) f'' + \frac{n-1}{r^2} f',
		\]
which shows that $f'''$ and $f'$ have the same sign when $f''=0$. In particular, $f''$ cannot change sign from negative to positive when $f'<0$. 
\end{proof}

As Lemma \ref{lem:small_facts_geodesics} shows, there are many situations where $f$ is concave down. By comparing with the Gauss area of enclosed triangles, we prove the following estimates, which will be used to bound $f$ in Lemma \ref{lem:f_bounded_above}. 
\begin{lemma}
\label{lem:area_triangle}
Let $f: I \to \R$ be a positive function with $f''\leq 0$ on the interval $(\xi_1, \xi_3) \subseteq I $ and such that the Gauss area under its graph is less than $\pi$, i.e. 
	\[
	\int_{\xi_1}^{\xi_3} f(r) \Bigl\{ 1 + \frac{n-1}{r^2}\Bigr\} \, dr \leq \pi. 
	\]
Then for any $\xi_2 \in [\xi_1, \xi_3]$, we have
	\begin{equation}
	\label{eq:GATriangle}
	\text{I}(\xi_2) + \text{II}(\xi_2) + \text{III}(\xi_2) \leq \pi,
	\end{equation}
where 
	\begin{align*}
	\text{I}(\xi_2)&:= {f(\xi_2) ( \xi_3 - \xi_1)}/{2}, \\
	\text{II}(\xi_2)&:=  (n-1){f(\xi_2)} \lim_{r\to \xi_2}   \left( \ln (r/ \xi_1) + {\xi_1}/{r} - 1\right)/({r-\xi_1}),\\
	\text{III}(\xi_2) &:= (n-1) {f(\xi_2)} \lim_{r\to \xi_2}  \left( \ln ({r}/{\xi_3}) + {\xi_3}/{r} - 1\right)/({\xi_3 - r}). 
	\end{align*}

\end{lemma}
The only reason for the limits is so we can take $\xi_2= \xi_1$ in {\it II} and $\xi_2=\xi_3$ in {\it III}. Note that {\it I} is the Euclidean area of $T$.

\begin{proof}
Because of the convexity of $f$, the region under the graph of $f$ contains the triangle $T$ with vertices $(\xi_1, 0), (\xi_2, f(\xi_2)), (\xi_3, 0)$. The Gauss area in $T$ is given by 
	\begin{multline*}
	 \int_{\xi_1}^{\xi_2} \frac{f(\xi_2)}{\xi_2- \xi_1} (r-\xi_1) \left[ 1 +\frac{(n-1)}{r^2} \right] dr 
		+ \int_{\xi_2}^{\xi_3} \frac{f(\xi_2)}{\xi_3-\xi_2} (\xi_3-r) \left[ 1 +\frac{(n-1)}{r^2} \right] dr,
	\end{multline*}
which gives the left-hand side of \eqref{eq:GATriangle} after integration.   
\end{proof}


\begin{lemma}
\label{lem:f_domain}
If $f: (a, b) \to \R$ is a maximally extended solution to \eqref{eq:GeodesicEquationf}, then $a < r_n < b$.
\end{lemma}
\begin{proof}
First, we show if $a<\rncyl$, then $b>\rncyl$. Suppose to the contrary that $a<\rncyl$ and $b \leq \rncyl$. We have that 
	\begin{equation}
	\label{eq:Limf}
	\lim_{x \to b^-} |f'(x)| = \infty.
	\end{equation}
If $b<\rncyl$ and $f$ stays bounded close to $b$, \eqref{eq:Limf} is inconsistent with \eqref{eq:GeodesicEquationf} as $f''(x)$ would have the opposite sign of $f'(x)$ as $x \to b^-$. If $b\leq \rncyl$ and $f$ blows up at $b$, the function $f$ has the same sign as $f'$ and we arrive to the same paradox. If $b = \rncyl$ and $f$ is bounded close to $b$, then the graph of $f(r)$ touches the cylinder $r \equiv \rncyl$ tangentially, which is impossible (because it would have to coincide with the cylinder). We conclude that $b>\rncyl$ when $a<\rncyl$. A similar argument shows $a<\rncyl$ when $b>\rncyl$. Therefore, $\rncyl \in (a,b)$.
\end{proof}

The finite Gauss area and the concavity now guarantee that $f$ is bounded. 

\begin{lemma}
\label{lem:f_bounded_above}
Let $f: [\varepsilon_0, R_0] \to \R$ be a solution to \eqref{eq:GeodesicEquationf} with $r_n \in (\varepsilon_0,R_0)$. Suppose in addition that $f$ is positive on $[\varepsilon_0,R_0]$ and the Gauss area under the graph of $f$ is at most $\pi$. Then
	\begin{equation}
	\label{eq:Boundf_LeftSide}
	f(r) \leq M_1:= \max \Big(f(\ve_0), \frac{2\pi}{\rncyl - \ve_0} \Big), \quad x \in [\varepsilon_0, r_n],
	\end{equation}
and 
	\begin{equation}
	\label{eq:Boundf_RightSide}
	f(r) \leq M_2:= \max \Big(f(R_0), \frac{2\pi}{R_0-\rncyl} \Big), \quad x \in [r_n,R_0].
	\end{equation}
\end{lemma}
In this proof, we use only the term $I(\xi_2)$ when we invoke Lemma \ref{lem:area_triangle}. 
\begin{proof}
Because $f$ is positive on $[\varepsilon_0,R_0]$, it follows from Lemma \ref{lem:small_facts_geodesics}(1) that $f$ does not have a local minimum on $[\varepsilon_0,R_0]$. In fact, $f''$ will be strictly negative at any point where $f'$ vanishes. Thus, if $f'(\rho_0) \leq 0$, then $f'(r) < 0$ for $r > \rho_0$. Similarly, if $f'(\rho_0) \geq 0$, then $f'(r) > 0$ for $r < \rho_0$.\\

{\bf Case 1: Suppose $f'(\varepsilon_0) \leq 0$.} Then $f$ is decreasing on $(\varepsilon_0,R_0]$ and we can take $M_1 = f(\ve_0)$. Lemma \ref{lem:small_facts_geodesics}(3) implies that $f''(r)<0$ for $r \in [r_n,R_0]$ and Lemma \ref{lem:area_triangle} with $\xi_1=\xi_2=r_n, \xi_3 = R_0$ gives us $f(r_n) \leq 2 \pi/(R_0- r_n)$. Thus, $f(r) \leq 2 \pi / (R_0 - r_n)$ for $r \in [r_n,R_0]$. \\

{\bf Case 2: Suppose $f'(R_0) \geq 0$.} This case is similar to Case 1 but here $f$ is increasing on $[\varepsilon_0,R_0)$. \\

{\bf Case 3: Suppose $f'(\varepsilon_0) >0$ and $f'(R_0) < 0$.} Then $f$ achieves a maximum in $(\varepsilon_0,R_0)$. Let $\rho_{max}$ denote the point where the maximum occurs. It follows from Lemma \ref{lem:small_facts_geodesics} and \eqref{eq:GeodesicEquationf} that $f''<0$ on $[\varepsilon_0,R_0]$ therefore we can apply Lemma \ref{lem:area_triangle} for $\xi_1=\ve_0, \xi_2 = \rho_{\max}$ and $\xi_3=R_0$. We obtain $f(\rho_{max}) \leq 2 \pi/(R_0 -\ve_0)$ therefore \eqref{eq:Boundf_LeftSide} and \eqref{eq:Boundf_RightSide} are true. \end{proof}

When the initial data for $f$  are small (see \eqref{eq:shooting_close}), its graph stays close to the plane $x \equiv 0$. This fact will be used in the next section to show that our curves $\gamma_i$ can not be close to such geodesic and thereby have to stay within a compact set. 

\begin{definition}
\label{def:EpsilonandR}
Let $L_0$ be the length of our initial rectangle $\mathcal R[a_0, \varphi(a_0), c_0]$. Let  $\eta$ be a small constant and $R$ a large constant  so that the graph of any positive function $h$ with domain $(1/R, R)$ and $h \leq \eta$ has length greater than $\frac{1}{2} L_0$. 
\end{definition}
Note that the constants exist because $L_0 < 2 L_{n,\pla}$. 
	\begin{proposition}
	\label{prop:GeodesicCloseToPlane}
	Let $R$ be as in Definition \ref{def:EpsilonandR}. There exists a constant $\delta >0$ so that the maximally extended solution $f:(a,b) \to \R$ to the graphical geodesic equation \eqref{eq:GeodesicEquationf} with initial conditions
		\begin{equation}
		\label{eq:shooting_close}
		0 \leq f(\rho_0) \leq \delta \text{ and } |f'(\rho_0)| \leq \delta, \quad \rho_0 \in (1/2, 2 r_n),
		\end{equation}
	has a domain containing $[-1/R, R]$ and presents one of the following two behaviors:
		\begin{enumerate}
		\item the graph $f([-1/R, R])$ stays above the $r$-axis and has length greater than $\frac{1}{2} L_0$, or
		\item the graph of  $f$ crosses the $r$-axis with finite slope at a point in $(1/R, R)$. 
		\end{enumerate}
	\end{proposition}
		

\begin{proof}
Thanks to the smooth dependence on initial data for \eqref{eq:GeodesicEquationf}, there exists a $\delta >0$ such that any geodesic $f$ with initial condition \eqref{eq:shooting_close} has a domain containing $(1/R, R)$ and 
	$
	|f(r)| \leq \eta$,  $ |f'(r)| \leq \eta$ for $ r \in (1/R, R)$,
where $\eta$ and $R$ are given in Definition \ref{def:EpsilonandR}. The two possible behaviors follow immediately.
\end{proof}

\subsection{Convergence to a shrinking doughnut}

We will see that if the curves $\gamma_i$ escape compact sets, they would have to converge to one of the geodesics described in Proposition \ref{prop:GeodesicCloseToPlane}, which would lead to a contradiction. 

Let us denote by $a(i)$ and $b(i)$ the two points where $\gamma_i$ intersects the $r$-axis, with the convention
	\[
	a(i)<\rncyl<b(i).
	\]
We can extract convergent subsequences $a(i_j)$ and $b(i_j)$ for which 
	\begin{align*}
	\lim_{j \to \infty} a(i_j) &= a_{\infty}, & a_{\infty} &\in [0, r_n]\\
	\lim_{j \to \infty} b(i_j) &= b_{\infty}, & b_{\infty} &\in [r_n, \infty].
	\end{align*} 
By Proposition \ref{prop:SymmetryGraphical}, the curves $\gamma_{i}$ restricted to the first quadrant are the graphs of functions of $r$. 
\begin{definition}
For each $i \in \mathbb N$, let $f_{i} : (a(i), b(i)) \to [0, \infty)$ be the function such that 
	\[
	\{ (r, f_{i}(r)), r \in (a(i), b(i))\} \subset \gamma_{i} (\S^1)
	\]
\end{definition}
Note that the $f_{i}$'s do not in general satisfy \eqref{eq:GeodesicEquationf}. However, we can extract a subsequence, also denoted by $f_{i}$, that converges in $C^1$ in any compact set to a geodesic $f$ thanks to Proposition \ref{prop:C1Convergence}. 
\begin{lemma}
\label{lem:a_infty}
$a_{\infty} \in (0, \rncyl)$. 
\end{lemma}

\begin{proof}
We argue by contradiction to exclude the possibilities $a_{\infty} = 0$ and $a_{\infty} = r_n$. \\

\noindent {\bf Case 1: Suppose $a_{\infty} = 0$.} We will show that the limit geodesic $f$ satisfies \eqref{eq:shooting_close} and obtain a contradiction. 
Choose $\ve$ a small positive number so that 
	\[
	\max\left( 8\ve, \frac{16 \ve}{\rncyl}, \frac{\pi}{\rncyl (- \ln(2 \ve) -1)} \right)< \min\left(\frac{\delta}{2}, \frac{1}{2R}\right),
	\]
with $\delta$ is as in Proposition \ref{prop:GeodesicCloseToPlane} and $R$ as in Definition \ref{def:EpsilonandR}. Consider the small box $B_{\ve} = [\ve, 2 \ve] \times [0, 4 \ve] \subset \R^2_+$. It has 
Gauss area larger than $\pi$ so all the $f_{i}$'s intersect $B_{\ve}$. Therefore the graph of $f$ also intersects $B_{\ve}$. In particular, $f(\ve_0) \leq 4 \ve$ for some $\ve_0 \in [\ve, 2 \ve]$ and we choose $M_1$ as in Lemma  \ref{lem:f_bounded_above}. 

Let us place ourselves in the compact set $K =  [\ve_0, 1/\ve_0] \times [0, M_1]$.  

(a) If $f$ is nonincreasing at some point in $B_{\ve} \cap K$, the function $f$ is nonincreasing all the way to $\rncyl$ by Lemma \ref{lem:small_facts_geodesics} (1). In particular, $0 \leq f(r) \leq 8 \ve$  for $r \in [\rncyl/2, \rncyl]$ and by the mean value theorem, 
	\[
	0\leq f(\rho_0) \leq 8 \ve, \quad  |f'(\rho_0)| \leq 16 \ve/\rncyl
	\]
for some  $\rho_0 \in ({\rncyl}/{2},\rncyl)$. By Proposition \ref{prop:GeodesicCloseToPlane}, this would mean the curves $\gamma_i$ converge smoothly in $K$ to a geodesic that is either too long or intersects the $x$-axis with finite slope. In either case, we have a contradiction. 


(b) Suppose now that $f(\ve_0) >0 $ and $f'(\ve_0) > 0$. By Lemma \ref{lem:small_facts_geodesics} (2) and Lemma \ref{lem:f_bounded_above}, we have that $f''<0$ on $(\ve_0, \rncyl)$. We can assume without loss of generality that $f(\rncyl/2) > f(\ve_0)$ (otherwise we are in a situation similar to the one where $f$ is nonincreasing). Lemma \ref{lem:area_triangle} with $\xi_1= \ve_0, \xi_2 =\xi_3= r_n/2$ gives
	\begin{align*}
	\pi \geq   (n-1) \frac{f(\rncyl/2)}{\rncyl/2-\ve_0} \left[\ln\left(\frac{\rncyl}{2\ve_0}\right) + \frac{2\ve_0}{\rncyl}-1\right] \geq \rncyl f(\rncyl/2) [-\ln(2 \ve_0) -1].
	\end{align*}
Hence, 
	$
	f(\rncyl/2) \leq \delta/2
	$
by the definition of $\delta$.
If $f' (\rncyl/2) \geq 0$, the convexity of $f$ and the fact that $f(\ve_0) \geq 0$ give $f'(\rncyl/2) \leq \frac{f(\rncyl/2)}{\rncyl/2-\ve} \leq \delta$. Similarly, if $f' (\rncyl/2)< 0$, the convexity and  $f (\rncyl) \geq 0$ imply $f'(\rncyl/2) \geq -\delta$. In either case, \eqref{eq:shooting_close} is true for $\rho_0 = \rncyl/2$ and leads to a contradiction. Therefore $a_{\infty} \neq 0$. \\

\noindent {\bf Case 2: Suppose $a_{\infty} = \rncyl$.} We consider the compact set $K =  [\rncyl -1, \rncyl+1] \times [-M, M]$ with $M$ so large that the double cover of the cylinder within $K$ has length greater than our initial rectangle. The curves $\gamma_i$ converge to the cylinder in $K$, which would contradict the fact that the length of each $\gamma_i$ is strictly less than the length of the double cover of the cylinder in $K$. 
\end{proof}
\begin{lemma}
\label{lem:b_infty}
$b_{\infty} \in (\rncyl, \infty)$. 
\end{lemma}
\begin{proof}
The proof is analogous to the one for $a_{\infty}$. Choose $\ve< R^{-1}$ so that 
	\[
	\frac{2 \pi}{r_n (2 \ve^{-1} - r_n)} \leq \delta
	\]
where $R$ is as in Definition \ref{def:EpsilonandR}. Consider the box $B_{\ve} = [\ve^{-1}, 2 \ve^{-1}] \times [0, 4\ve]$, which has Gauss area greater than $2 \pi$. As in Case 1(a) of Lemma \ref{lem:a_infty}, $f'\geq 0$ for some point in $B_{\ve}$  leads to a contradiction. Therefore $f(2\ve^{-1})>0$,  $f'(2 \ve^{-1}) <0$, and we have $f''<0$ on $(\rncyl, 2 \ve^{-1})$ by Lemma \ref{lem:small_facts_geodesics} (3). We apply Lemma \ref{lem:area_triangle} with $\xi_1= r_n, \xi_2 = 2 r_n, \xi_3 = 2 \ve^{-1}$ and obtain
	\[
	2 \pi \geq  f(2 r_n) (2 \ve^{-1} -  \rncyl). 
	\]
If $f'(2 \rncyl) \leq 0$, we can take $\rho_0 = 2\rncyl$ in \eqref{eq:shooting_close}. If $f'(2 \rncyl)>0$, we have $f'(2\rncyl) \leq f(2\rncyl)/\rncyl \leq \delta$. In either case, the curves $\gamma_i$ converge in $K = [1/R, R] \times [0,M_2]$ to a geodesic with initial condition \eqref{eq:shooting_close}, which gives us again a contradiction. 
\end{proof}

Consider now the compact set 
	\[
	K = [a_\infty-1, b_{\infty} + 1] \times [-M, M],
	\]
where $M$ is as in Lemma \ref{lem:f_bounded_above} with $\ve_0 = a_{\infty}$ and $R_0 = b_{\infty}$.  
Our subsequence $\gamma_{i}$ converges to a geodesic $\gamma_{\infty}$ in $K$. Because the intersections of the $\gamma_{i}$'s with the $r$-axis are eventually in $K$, the intersections of $\gamma_{\infty}$ with the $r$-axis are in $K$. The curve $\gamma_{\infty}$ is a connected $C^1$ curve that satisfies the geodesic equation in $K$, therefore it is smooth. It is curve generating the promised self-shrinking doughnut.

\bigskip


\section{Appendix}


Let us recall that $
	h_n (s) =  e^{-s^2/4} s^n$ and $H_n(s) = h_n(s-2) + h_n(s+2)$.

\begin{proposition} The inequality 
	\begin{equation}
	\tag{\ref{eq:pointwise}}
	 (2 r_{n+4} - t) H_n(r_{n+4}-t) < 2n ( 2 r_{n+2}+t)  H_{n-2}(r_{n+2}+t)
	\end{equation}
 for $0<t<r_{n+4} -r_{n+3}$ and $n \geq 7$ is at its worst when $t=0$, i.e. 
 \begin{gather*}
	(2 r_{n+4} - t) H_n(r_{n+4}-t)  \leq 2 r_{n+4} H_n(r_{n+4}) \\
	4n r_{n+2}  H_{n-2}(r_{n+2})  \leq 2n ( 2 r_{n+2}+t)  H_{n-2}(r_{n+2}+t) 
	\end{gather*}

 \end{proposition}
\begin{proof}
We have
	\begin{align*}
	h_n (s) & =  e^{-s^2/4} s^n, \\
	h'_n(s) & = e^{-s^2/4} s^{n-1} \left( n - \tfrac{1}{2} s^2 \right),\\
	h''_n(s) & = e^{-s^2/4} s^{n-2} \left( n(n-1) - \left(n+ \tfrac{1}{2}\right) s^2 + \tfrac{1}{4} s^4 \right).
	\end{align*}
By the quadratic formula, $h''(s) \geq 0$ if and only if $s^2 \leq s_0^2:= 2n+1 - \sqrt{8n+1}$ or $s^2 \geq s_1^2:= 2n+1 + \sqrt{8n+1}$. For $H''_n$, we are concerned about $s \in (r_{n+3} - 2, r_{n+5} -2) \cup (r_{n+3}+2, r_{n+5}+2)$  and see that 
	\[
	r_{n+5} - 2 < s_0, \quad s_1< r_{n+3}+2
	\]
for $n \geq 7$. So at least $H''_n (s) >0$ for $s \in (r_{n+3}, r_{n+4})$ and $n \geq 7$.  We will now show that $H'_n(r_{n+3})>0$, or equivalently, that 
	\[
	Q_n:= \frac{-h'_n(r_{n+3} +2)}{h'_n(r_{n+3}-2)}<1. 
	\]
Rearranging terms, we have
	\begin{equation}
	\label{eq:QAgain}
	 Q_n = e^{-2 r_{n+3}} \left(\frac{r_{n+3}+2}{r_{n+3}-2}\right)^n =e^{-2 r_{n+3}} \left(1+ \frac{4}{r_{n+3}-2}\right)^{(r_{n+3}^2-4)/2}.
	 \end{equation}
One can prove by taking $\log$, limits, and derivatives that
	\begin{align*}
	f(x) &= \left(1+\frac{a}{x}\right)^{x^2} \nearrow e^{ax}e^{-\frac{a^2}{2}}, \quad a>0\\
	g(x) &= \left(1-\frac{a}{x}\right)^{x^2} \searrow e^{-ax}e^{-\frac{a^2}{2}}, \quad a>0,
	\end{align*}
as $ x \to \infty$. The last factor in \eqref{eq:QAgain} can be estimated by 
	\[
	 \left(1+ \frac{4}{r_{n+3}-2}\right)^{\frac{(r_{n+3}-2)^2}{2}+2(r_{n+3}-2) + 4}   \leq  e^{2(r_{n+3}-2) } \left(1+ \frac{4}{r_{n+3}-2}\right)^4
	\]
and $Q_n<1$ for $n \geq 8$ because $1+ \frac{4}{r_{n+3}-2} < e$ when $n \geq 8$. We have $Q_7 <1$ by checking numerically from the definition. 
\end{proof}

\bigskip


\def\cprime{$'$}

\end{document}